\documentclass{article}
\usepackage{graphicx} 
\usepackage[total={6in, 9in}]{geometry}
\usepackage{amsfonts}
\usepackage{amsthm}
\usepackage{amssymb}
\usepackage{amsmath}
\usepackage{enumitem}
\usepackage[colorlinks=true,allcolors=blue]{hyperref}

\newtheorem{THM}{Theorem}
\newtheorem{LEM}[THM]{Lemma}
\newtheorem{CON}[THM]{Conjecture}
\newtheorem{CLA}{Claim}
\newtheorem*{unCLA}{Claim}
\theoremstyle{definition}
\newtheorem{CASE}{Case}

\newcommand{\CC}{\mathcal{C}}
\def\even{_{\fam0 even}}
\def\adm{_{\fam0 adm}}
\def\Ce{\CC\even}
\def\Ca{\CC\adm}
\def\Fe{F'\even}
\def\Fa{F'\adm}
\def\Xa{X\adm}
\def\Ya{Y\adm}
\def\Tp{T^*}
\def\Ct{{\widetilde C}}
\def\Xt{{\widetilde X}}
\def\Yt{{\widetilde Y}}
\def\xt{{\widetilde x}}
\def\yt{{\widetilde y}}

\title{Spanning weakly even trees of graphs}
\author{%
  Jiangdong Ai\footnote{School of Mathematical Sciences and LPMC, Nankai University, Tianjin 300071, P.R. China. Email: \texttt{jd@nankai.edu.cn}.}%
\and
  M.~N. Ellingham\footnote{Department of Mathematics, Vanderbilt University, 1326 Stevenson Center, Nashville, Tennessee 37240, USA. Email: \texttt{mark.ellingham@vanderbilt.edu}. Supported by Simons Foundation award MPS-TSM-00002760 and ARIS grant BI-US/22-24-77.}%
\and
  Zhipeng Gao\footnote{ School of Mathematics and Statistics, Xidian University, Xi’an 710126, P.R. China. Email: \texttt{gaozhipeng@xidian.edu.cn}.}%
\and
  Yixuan Huang\footnote{Department of Mathematics, Vanderbilt University, 1326 Stevenson Center, Nashville, Tennessee 37240, USA. Email: \texttt{yixuan.huang.2@vanderbilt.edu}.}%
\and
  Xiangzhou Liu\footnote{Department of Mathematics, Tiangong University, Tianjin 300071, P.R. China. Email: \texttt{i19991210@163.com}.}%
\and
  Songling Shan\footnote{Department of Mathematics and Statistics, Auburn University,  Auburn, Alabama 36849, USA. Email: \texttt{szs0398@auburn.edu}. Supported by NSF grant DMS-2345869.}%
\and
  Simon \v Spacapan\footnote{University of Maribor, FME, Maribor, Slovenia and 
  IMFM, Ljubljana, Slovenia. Email: \texttt{simon.spacapan@um.si}. Supported by ARIS program P1-0297, project N1-0218, and grant BI-US/22-24-77.}%
\and
 Jun Yue\footnote{ Department of Mathematics, Tiangong University, Tianjin 300071, P.R. China. Email: \texttt{yuejun06@126.com}.}%
}

\date{11 October 2024}

\begin{document}

\maketitle

\begin{abstract}
Let $G$ be a graph (with multiple edges allowed) and let $T$ be a tree in $G$.
We say that $T$ is \emph{even} if every leaf of $T$ belongs to the same part of the bipartition of $T$, and that 
$T$ is \emph{weakly even} if every leaf of $T$ that has maximum degree in $G$ belongs to the same part of the bipartition of $T$. We confirm two recent conjectures of Jackson and Yoshimoto by showing that every connected graph that is not a regular bipartite graph has a spanning weakly even tree. 
\bigskip 

\end{abstract}

\section{Introduction}

In this paper graphs are finite and may contain multiple edges but not loops.
We use $uv$ to denote an edge from $u$ to $v$; if there is more than one such edge, which edge we mean will either not matter or be clear from context.
Let $T$ be a tree in a graph $G$.
We say that $T$ is \emph{even} if all leaves of $T$ belong to the same part of the bipartition of $T$.
More generally, we say that $T$ is \emph{weakly even} if all leaves of $T$ that have maximum degree in $G$ belong to the same part of the bipartition of $T$.

For our proofs it is convenient to consider a specific ordered bipartition $(X,Y)$ of a tree $T$ in $G$ and to insist that the leaves of $T$ with maximum degree in $G$ belong to $X$.  We introduce some appropriate terminology.
Given an ordered bipartition $(X,Y)$ of a bipartite graph $H$, a vertex of $H$ is \emph{type-$0$} or \emph{type-$1$} if it belongs to $X$ or $Y$, respectively.
If $w \in V(H)$ and $\lambda \in \{0,1\}$, the \emph{$(w,\lambda)$-bipartition} of $H$ is the bipartition of $H$ for which $w$ has type $\lambda$.  If $H$ is equipped with this bipartition, we say $H$ is a \emph{$(w,\lambda)$-graph} (or \emph{$(w,\lambda)$-tree}, \emph{$(w,\lambda)$-cycle}, etc., as appropriate).
A tree $T \subseteq G$ with ordered bipartition $(X,Y)$ is \emph{even} if all leaves of $T$ are type-$0$ (belong to $X$), and \emph{weakly even} if all leaves of $T$ that have maximum degree in $G$ are type-$0$.

S.~Saito asked which regular connected graphs have a spanning even tree. 
Jackson and Yoshimoto~\cite{Jackson-Yoshimoto} obtained the following partial answer to the question.

\begin{THM}[\cite{Jackson-Yoshimoto}]
 \label{thm:j-y}
 Suppose $G$ is a regular nonbipartite connected graph that has a 2‐factor,
$w \in  V (G)$ and $\lambda \in\{0, 1\}$. Then G has a spanning even $(w,\lambda)$-tree.
\end{THM}

They  conjectured that connected regular bipartite graphs are the only connected graphs that do not have a spanning even tree.

\begin{CON}
    \label{conj:regular-even-sp}
    Every regular nonbipartite connected graph has a spanning even tree.
\end{CON}

As an extension of Conjecture~\ref{conj:regular-even-sp}, Jackson and Yoshimoto~\cite{Jackson-Yoshimoto}  also posed the following conjecture.

\begin{CON}
    \label{conj:weak-even-sp}
    Every connected graph that is not a regular bipartite graph has a spanning  weakly even tree.
\end{CON}

Conjecture~\ref{conj:weak-even-sp} implies Conjecture~\ref{conj:regular-even-sp}: in a regular connected graph, every vertex has maximum degree, and so a spanning tree is weakly even if and only if it is  even.
In this paper we confirm Conjecture~\ref{conj:weak-even-sp} and hence Conjecture~\ref{conj:regular-even-sp}.
Our result combines work by Ai, Gao, Liu, and Yue \cite{AGLY24arxiv}, and by Ellingham, Huang, Shan, and \v{S}pacapan \cite{EHSS24arxiv}.

Most of the work to confirm Conjecture~\ref{conj:weak-even-sp} is in the proof of the following theorem, which we postpone to the next section.

\begin{THM}\label{thm:spanning weak good tree}
Let $G$ be a $2$-edge-connected graph that is not regular bipartite, $w \in V(G)$ and $\lambda \in \{0,1\}$.
Then $G$ has a spanning weakly even $(w,\lambda)$-tree.
\end{THM}

Using Theorem \ref{thm:spanning weak good tree} we can prove  Theorem \ref{thm:general spanning weak good tree}, which verifies Conjecture \ref{conj:weak-even-sp}.

\begin{THM}\label{thm:general spanning weak good tree}
Let $G$ be a connected graph that is not regular bipartite, $w \in V(G)$ and $\lambda \in \{0,1\}$.
Then $G$ has a spanning weakly even $(w,\lambda)$-tree.
\end{THM}

\begin{proof}
We proceed by induction on $|V(G)|$.  If $|V(G)| \le 2$ then $G$ is regular bipartite, and the theorem holds vacuously.  Therefore, we may assume that $|V(G)| \ge 3$, which implies $\Delta(G) \ge 2$, and that the theorem holds for graphs of smaller order than $G$.  If $G$ is $2$-edge-connected, then the theorem holds by Theorem \ref{thm:spanning weak good tree}, so we may assume that $G$ has a cutedge $x_1 x_2$.  Let $G_1$ and $G_2$ be the components of $G-x_1x_2$, with $x_1 \in V(G_1)$, $x_2 \in V(G_2)$.

\begin{unCLA}
Let $i \in \{1,2\}$, $w_i \in V(G_i)$, and $\lambda_i \in \{0,1\}$.  Then there is a spanning $(w_i, \lambda_i)$-tree $T_i$ of $G_i$ that is weakly even in $G$ except possibly at $x_i$ (i.e., no vertex of $T_i$ except possibly $x_i$ is a type-$1$ leaf of $T_i$ with maximum degree in $G$).
\end{unCLA}

\begin{proof}[Proof of Claim]
Note that all vertices of $G_i$ have the same degree in $G_i$ as in $G$, except $x_i$.
If $G_i$ is regular bipartite then $\Delta(G_i) < \Delta(G)$.  Hence any spanning $(w_i,\lambda_i)$-tree $T_i$ of $G_i$ is weakly even in $G$ except possibly at $x_i$.  If $G_i$ is not regular bipartite, then by the induction hypothesis there is a spanning $(w_i,\lambda_i)$-tree $T_i$ of $G_i$ that is weakly even in $G_i$, and hence weakly even in $G$ except possibly at $x_i$.
\end{proof}

We may assume that $w \in V(G_1)$.  By the Claim $G_1$ has a spanning $(w,\lambda)$-tree $T_1$ that is weakly even in $G$ except possibly at $x_1$.  Let $\lambda_2$ be the type opposite to the type of $x_1$ in $T_1$.  By the Claim, $G_2$ has a spanning $(x_2,\lambda_2)$-tree $T_2$ that is weakly even in $G$ except possibly at $x_2$.  The bipartitions of $T_1$ and $T_2$ agree with the $(w,\lambda)$-bipartition of $T=T_1 \cup T_2 \cup \{x_1x_2\}$.  If either $x_1$ or $x_2$ is a leaf of $T$ then it has degree $1 < \Delta(G)$ in $G$, and all other leaves of $T$ satisfy the weakly even condition in $G$, so $T$ is a spanning weakly even tree in $G$.
\end{proof}

\section{Proof of Theorem \ref{thm:spanning weak good tree}}

We start with some preliminaries. We assume that every cycle in a graph $G$ has a fixed orientation.  When discussing a particular cycle $C$ and $u, v \in V(C)$ we use $u^-$ and $u^+$ to denote the immediate predecessor and successor, respectively, of $u$ on $C$, and $uCv$ to mean the subpath of $C$ from $u$ to $v$ following the orientation of $C$ ($uCv$ is a single vertex if $u=v$).
A spanning subgraph $H$ of $G$ is a \emph{weak 2-factor} if each component of $H$ is either a cycle or a path (possibly a single vertex), and the endvertices of the path components of $H$ have  degree less than $\Delta(G)$ in $G$.
For any positive integer $k$, let  $[k]=\{1,\ldots, k\}$.

\begin{THM}[\cite{Bae37, MR1670596, Pet1891}]\label{thm:2-factor}
If $G$ is a connected $r$-regular graph, $r \ge 2$, and $G$ is $2$-edge-connected or has at most $r-1$ cutedges, then $G$ has a $2$-factor.
\end{THM}

The following generalizes \cite[Lemma 6]{Jackson-Yoshimoto}.

\begin{LEM}\label{lem:Y-out-edge}
    Let  $G$ be a connected graph that is not regular  bipartite, and $Y \subseteq V(G)$ an independent set in $G$. Suppose that all   vertices of $Y$ have maximum degree in $G$. Then for any $X \subseteq V(G) \setminus Y$ with $|X| \le |Y|$, we have $E(Y,V(G)\setminus (X\cup Y)) \ne \emptyset$.
\end{LEM}

\begin{proof} 
Let $G,X$ and $Y$ be as given  in the lemma. If $E(Y, V(G) \setminus (X \cup Y)) = \emptyset$, then
\begin{align*}
    |E(X,Y)| &= \sum_{y \in Y } d(y) = \Delta(G) |Y|
     \ge \Delta(G) |X| \ge \sum_{x \in X} d(x) \\
    & = |E(X,Y)| + 2 |E(X)| + |E(X,V(G) \setminus (X \cup Y))| \ge |E(X,Y)|.
\end{align*}
Therefore, $|E(X)| = |E(X, V(G) \setminus (X\cup Y))| = 0$, $|X| = |Y|$, and $d(x) = \Delta(G)$ for all $x \in X$.
This implies that $G$ is a $\Delta(G)$-regular bipartite graph, a contradiction. 
\end{proof} 

Now we can prove Theorem \ref{thm:spanning weak good tree}.  The proof uses a similar argument to the original proof of Theorem \ref{thm:j-y} by Jackson and Yoshimoto.

\begin{proof}[Proof of Theorem~\ref{thm:spanning weak good tree}]
Let $G$ be a $2$-edge-connected graph that is not  regular bipartite, and 
let $w\in V(G)$ and $\lambda\in \{0,1\}$ be given.
If $G$ is regular, then since $G$ is $2$-edge-connected it has a $2$-factor by Theorem~\ref{thm:2-factor}, and hence a spanning even tree with $w$ of type $\lambda$ by Theorem \ref{thm:j-y}.
Assume therefore that $G$ is not a regular graph. Since $G$ is $2$-edge-connected, it follows that $\Delta(G) \ge 3$. 

We will find a weak 2-factor in $G$ and then construct a spanning weakly even tree by using most of the edges of  the weak 2-factor and some other edges.

\begin{CLA}\label{claim:weak-2-factor}
The graph $G$ has a weak 2-factor.
\end{CLA}
\begin{proof}  Let $G'$ be another copy of $G$. 
For each $v\in V(G)$ with $d_G(v)<\Delta(G)$, add $\Delta(G)-d_G(v)$
edges joining $v$ and the copy of $v$ in $G'$. Denote by $G^*$
the resulting multigraph. 
Then $G^*$ is $\Delta(G)$-regular and has at most one cutedge. 
By Theorem~\ref{thm:2-factor}, 
$G^*$ has a $2$-factor $F^*$.
Let $F$ be the spanning subgraph of $G$ such that $E(F) = E(F^*) \cap E(G)$.
Since $F^*$ is a 2-factor,  each component of $F$ is either a cycle or a path.
Moreover, each endvertex of a path component of $F$ is incident with an edge in $E(F^*) \setminus E(G)$ and thus has degree less than $\Delta(G)$.
Therefore $F$ is a weak $2$-factor of $G$.
\end{proof} 

Fix a weak 2-factor $F$ of $G$.
A tree in $G$ is \emph{good} (with respect to $F$) if its vertex set is the union of vertex sets of some components of $F$.
Let $T$ be a good weakly even $(w,\lambda)$-tree in $G$ of maximum order; if no such tree exists, let $T$ be the null graph.
We claim that $T$ is a spanning tree of $G$.
We suppose that $V(T) \ne V(G)$ and show that this always leads to a contradiction.

\begin{CLA}\label{claim:w-cycle-even}
Suppose that $T$ is null.
Then the component $C_0$ of $F$ containing $w$ is an even cycle.
If $(X_0, Y_0)$ is the $(w,\lambda)$-bipartition of $C_0$, then all vertices of $Y_0$ have maximum degree in $G$, and $Y_0$ is an independent set in $G$.
\end{CLA}

\begin{proof}
If $C_0$ is a path, then both ends of $C_0$ have degree less than $\Delta(G)$ in $G$.  Therefore, $C_0$ is a good even $(w,\lambda)$-tree, which contradicts the choice of $T$.
If $C_0$ is an odd cycle, then $w^+C_0w$ is a good even $(w,0)$-tree, and $w^{++}C_0w^+$ is a good even $(w,1)$-tree.  Therefore, there is a good even $(w,\lambda)$-tree, which contradicts the choice of $T$.
Thus, $C_0$ must be an even cycle.

If there is $u \in Y_0$ that has degree less than $\Delta(G)$ in $G$, then $T = u^+C_0u$ is a good weakly even $(w,\lambda)$-tree (since $u^+\in X_0$ is the only leaf of $T$ that possibly has degree $\Delta(G)$ in $G$), contradicting the choice of $T$.
Therefore, all vertices of $Y_0$ have degree $\Delta(G)$ in $G$.

Suppose that $G$ has an edge $uv$ joining $u, v \in Y_0$.
Consider $T_1 = u^+C_0 u^- \cup \{uv\}$ and $T_2 = v^+ C_0 v^- \cup \{uv\}$.
If $w \ne u$ then $T_1$ is a good even $(w,\lambda)$-tree, and if $w = u$ then $T_2$ is a good even $(w,\lambda)$-tree.
Either situation contradicts the choice of $T$.
Thus, $Y_0$ is an independent set in $G$.
\end{proof}

If $T$ is null, let $C_0$ be the  component (even cycle) of $F$ that contains $w$, and let $(X_0,Y_0)$ be the $(w,\lambda)$-bipartition of $C_0$.
If $T$ is nonnull, let $(X_0,Y_0)$ be the $(w,\lambda)$-bipartition of $T$.  

 \begin{CLA}\label{claim:Y_0-edges}
 	Suppose that $T$ is nonnull. 
 \begin{enumerate}[label=\rm(\alph*)]\setlength{\itemsep}{0pt}
       \item Then $E_G(X_0, V(G)\setminus V(T)) =\emptyset$ and $E_G(Y_0, V(G)\setminus V(T)) \ne \emptyset$.
      \item Suppose that $y_0z \in E(G)$ with $y_0 \in Y_0$ and $z \notin V(T)$.
Then the component $C$ of $F$ containing $z$ is an even cycle.
If $(X, Y)$ is the $(z,0)$-bipartition of $C$ (so that $z \in X$), then all vertices of $Y$ have maximum degree in $G$, and $Y$ is an independent set in $G$.
  \end{enumerate}
 \end{CLA}
 
\begin{proof} 
(a)
If $E_G(X_0,V(G)\setminus V(T)) \ne \emptyset$, then there exists $x_0 z \in E(G)$ such that $x_0 \in X_0$ and $z \notin V(T)$. 
Let $C$ be the component of $F$ containing $z$.
If $C$ is a path, then $T' = T \cup \{x_0z\} \cup C$ is a good weakly even $(w,\lambda)$-tree.
Indeed, each leaf of $T'$ that is not a leaf of $T$  has degree less than $\Delta(G)$ in $G$ (since $F$ is a weak $2$-factor).
If $C$ is an odd cycle, then $T \cup \{x_0z\} \cup z^{++}Cz^+$ is a good weakly even $(w,\lambda)$-tree.
If $C$ is an even cycle, then $T \cup \{x_0z\} \cup z^+Cz$ is a good weakly even $(w,\lambda)$-tree.  All three situations contradict  the maximality of $T$, and therefore $E_G(X_0,V(G)\setminus V(T)) = \emptyset$.

If $E_G(Y_0,V(G)\setminus V(T)) = \emptyset$  then $E_G(V(T),V(G)\setminus V(T)) = \emptyset$, which (since $V(T) \ne \emptyset$ and $V(T) \ne V(G)$) contradicts the fact that $G$ is connected.
Hence, $E_G(Y_0,V(G)\setminus V(T)) \ne \emptyset$.

(b)
If $C$ is a path, then each leaf of $T'= T \cup \{y_0z\} \cup C$ that is not a leaf of $T$ is an endvertex of $C$ and so has degree less than $\Delta(G)$ in $G$.  Thus, $T'$ is a good weakly even $(w,\lambda)$-tree.
If $C$ is an odd cycle, then $T \cup \{y_0z\} \cup z^+Cz$ is a good weakly even $(w,\lambda)$-tree. 
In both cases we contradict the maximality of $T$.
Hence, $C$ is an even cycle.

If there is $u \in Y$ that has degree less than $\Delta(G)$ in $G$, then $T' = T \cup \{y_0 z\} \cup u^+Cu$ is a good weakly even $(w,\lambda)$-tree (since $u$ is the only vertex of $Y$ that is a leaf of $T'$ but not of $T$), contradicting the maximality of $T$.
Therefore, all vertices of $Y$ have degree $\Delta(G)$ in $G$.

Suppose that $G$ has an edge $uv$ joining $u, v \in Y$.  Let $T' = T \cup \{y_0z, uv\} \cup u^+Cu^-$.
Then $T'$ is a good weakly even $(w,\lambda)$-tree because every vertex that is not a leaf of $T$ but possibly a leaf of $T'$ (namely $u^-$, $u$, $u^+$) is type-$0$.  This contradicts the maximality of $T$.
Thus, $Y$ is an independent set in $G$.
\end{proof}

Since $T$ is a good tree, $F' = F-V(T)$ is a union of components of $F$.
Let $\Ce$ be the set of all even cycles of $F'$, and $\Fe = \bigcup_{C \in \Ce} C$.
If $T$ is null then $C_0 \subseteq \Fe$, and we say a bipartition $(X,Y)$ of $\Fe$ is \emph{consistent} if $X_0 \subseteq X$ and $Y_0 \subseteq Y$.  If $T$ is nonnull, every bipartition of $\Fe$ is considered to be consistent.
Let $C_1, C_2, \dots, C_\ell$ be a sequence of $\ell\geq 1$ distinct cycles in $\Ce$.  Given such a sequence, we define $X_i = X \cap V(C_i)$ and $Y_i = Y \cap V(C_i)$ for $i \in [\ell]$.  The sequence is \emph{admissible} with respect to a consistent bipartition $(X,Y)$ of $\Fe$ if (1) either $T$ is null and $C_1 = C_0$, or $T$ is nonnull and $E(Y_0, X_1) \ne \emptyset$, and (2) $E(Y_i, X_{i+1}) \ne \emptyset$ for $i \in [\ell-1]$.  A cycle $C \in \Ce$ is \emph{admissible} if it is the final (equivalently, any) cycle of an admissible sequence.

Choose a consistent bipartition $(X,Y)$ of $\Fe$ such that the number of admissible cycles with respect to $(X,Y)$
is as large as possible, and let $\Ca$ be the set of admissible cycles with respect to $(X,Y)$.
If $T$ is null then $C_0\in \Ca$, and if $T$ is nonnull, then $\Ca$ is nonempty by Claim \ref{claim:Y_0-edges}.
Let $\Fa = \bigcup_{C \in \Ca} C$, $\Xa = X \cap V(\Fa)$, and $\Ya = Y \cap V(\Fa)$.
By the maximality of $\Ca$, if  $E(V(\Fa) \cap Y, V(C)) \ne \emptyset$ for a component $C$ of $F'$ that is vertex-disjoint from $\Fa$, then $C$ must be an odd cycle or a path.

\begin{CLA}\label{claim:cycles-maximum-degree-Y-vertices}
All vertices in $\Ya$ have maximum degree in $G$.
\end{CLA}

\begin{proof}
Suppose that the claim is false and let $C_1, C_2, \ldots, C_\ell$ be a shortest admissible sequence with respect to $(X,Y)$ for which some vertex $z$ in $Y \cap C_\ell$ does not have maximum degree in $G$.
Note that, by Claims \ref{claim:w-cycle-even} and \ref{claim:Y_0-edges}, all vertices of $Y_1$ have maximum degree in $G$, so $\ell\ge 2$.
There is $y_ix_{i+1} \in E(G)$ with $y_i \in Y_i$ and $x_{i+1} \in X_{i+1}$ for $i\in [\ell-1]$, and for $i=0$ if $T$ is nonnull.
Let $\Tp$ be null if $T$ is null, and $T \cup \{y_0 x_1\}$ if $T$ is nonnull.
Define
$$\textstyle T' =  \Tp \cup (\bigcup_{i=1}^{\ell-1} y_i^+ C_{i} y_i x_{i+1} ) \cup z^+ C_{\ell} z .$$
Then $T'$ is a good weakly even $(w,\lambda)$-tree since $z$, which has degree less than $\Delta(G)$, is the only vertex in $Y$ that is a leaf of $T'$ and not a leaf  of $T$.  This contradicts the maximality of $T$.
\end{proof}

We now consider two cases.

\begin{CASE} Suppose $\Ya$ is an independent set in $G$.

In this case, Lemma~\ref{lem:Y-out-edge} implies that there is an edge $uz$ of $G$ joining $u \in \Ya$ and $z \in V(G)\setminus V(\Fa)$.
Then there is an admissible sequence $C_1, C_2, \dots, C_\ell$ with $u \in V(C_\ell)$.
There is $y_ix_{i+1} \in E(G)$ with $y_i \in Y_i$ and $x_{i+1} \in X_{i+1}$ for $i\in [\ell-1]$, and for $i=0$ if $T$ is nonnull.
Let $\Tp$ be null if $T$ is null, and $T \cup \{y_0 x_1\}$ if $T$ is nonnull.

If $z \notin V(T)$, let $C$ be the component of $F'$ that contains $z$.  Then $C$ is an odd cycle or a path, by the maximality of $\Ca$.  Let $Q=z^+ C z$ if $C$ is an odd cycle and $Q=C$ if $C$ is a path.
Define
$$\textstyle T' =  \Tp \cup (\bigcup_{i=1}^{\ell-1} y_i^+ C_{i} y_i x_{i+1}) \cup u^+C_\ell uz \cup Q .$$
Then $T'$ is a good weakly even $(w,\lambda)$-tree.  This contradicts the maximality of $T$.

If $z\in V(T)$, this implies that $T$ is nonnull. Since $u$ is a neighbor of $z$ not in $V(T)$, by Claim~\ref{claim:Y_0-edges}, $z \in Y_0$.  Define
$$\textstyle T' =  T \cup \{y_0x_1, uz\} \cup (\bigcup_{i=1}^{\ell-1} y_i^+ C_{i} y_i x_{i+1}) \cup u^+C_\ell u^-.$$
Then $T'$ is a good weakly even $(w,\lambda)$-tree because every vertex that is not a leaf of $T$ but possibly a leaf of $T'$ (including $u$) is type-$0$.  This contradicts the maximality of $T$.
\end{CASE}

\begin{CASE} Suppose $\Ya$ is not an independent set in $G$.

Then there is an edge $yz$ where $y, z \in \Ya$.  Suppose that $y$ is a vertex of $C \in \Ca$ and $z$ is a vertex of $\Ct \in \Ca$, where possibly $C = \Ct$.
Let $C_1, C_2, \dots, C_\ell$ be an admissible sequence with $C = C_\ell$.
There is $y_ix_{i+1} \in E(G)$ with $y_i \in Y_i$ and $x_{i+1} \in X_{i+1}$ for $i\in [\ell-1]$, and for $i=0$ if $T$ is nonnull.
Let $\Ct_1, \Ct_2, \dots, \Ct_k$ be an admissible sequence with $\Ct = \Ct_k$, and let $\Xt_j = X \cap V(\Ct_j)$, $\Yt_j = Y \cap V(\Ct_j)$ for $j \in [k]$.
There is $\yt_j\xt_{j+1} \in E(G)$ with $\yt_j \in \Yt_j$ and $\xt_{j+1} \in \Xt_{j+1}$ for $j\in [k-1]$, and for $j=0$ if $T$ is nonnull.
Let $\Tp$ be null if $T$ is null, and $T \cup \{y_0 x_1\}$ if $T$ is nonnull.

Suppose first that $C$ and $\Ct$ belong to a common admissible sequence.  Without loss of generality we may suppose that $\Ct \in \{C_1, C_2, \dots, C_\ell\}$.  Define
$$\textstyle T' =  \Tp \cup \{yz\} \cup (\bigcup_{i=1}^{\ell-1} y_i^+ C_{i} y_i x_{i+1}) \cup y^+C_\ell y^-.$$
Then $T'$ is a good weakly even $(w,\lambda)$-tree because every vertex that is not a leaf of $T$ but possibly a leaf of $T'$ (including $y$) is type-$0$.  This contradicts the maximality of $T$.

Next, suppose that $\{C_1, C_2, \dots, C_\ell\} \cap \{\Ct_1, \Ct_2, \dots, \Ct_k\} = \emptyset$.  Then $T$ is nonnull. Let
$$\textstyle T' = 
    T \cup \{y_0x_1, \yt_0\xt_1, yz\}
    \cup (\bigcup_{i =1}^{\ell-1} y_i^+ C_{i} y_i x_{i+1})
    \cup (\bigcup_{j=1}^{k-1} \yt_{j}^+ \Ct_j \yt_j \xt_{j+1})
    \cup y^+ C_\ell y^- \cup z^+ \Ct_k z.
$$
Again $T'$ is a good weakly even $(w,\lambda)$-tree because every vertex that is not a leaf of $T$ but possibly a leaf of $T'$ (including $y$) is type-$0$.  This contradicts the maximality of $T$.

Finally, suppose that $\{C_1, C_2, \dots, C_\ell\} \cap \{\Ct_1, \Ct_2, \dots, \Ct_k\} \ne \emptyset$, but $C = C_\ell \notin \{\Ct_1, \Ct_2, \dots, \Ct_k\}$ and $\Ct = \Ct_k \notin \{C_1, C_2, \dots, C_\ell\}$.
Let $b \in [k-1]$ be the largest integer such that $\Ct_b \in \{C_1, C_2, \dots, C_\ell\}$.  Then $\Ct_b = C_a$ for some $a \in [\ell-1]$ and $\{C_1, C_2, \dots, C_\ell\} \cap \{\Ct_{b+1}, \Ct_{b+2}, \dots, \Ct_k \} = \emptyset$.
Let
$$\textstyle T' =
    \Tp \cup\{\yt_b \xt_{b+1}, yz\}
    \cup (\bigcup_{i =1}^{\ell-1} y_i^+ C_{i} y_i x_{i+1})
    \cup (\bigcup_{j=b+1}^{k-1} \yt_{j}^+ \Ct_j \yt_j \xt_{j+1})
    \cup y^+C_\ell y^- \cup z^+\Ct_k z.$$
Once more $T'$ is a good weakly even $(w,\lambda)$-tree because every vertex that is not a leaf of $T$ but possibly a leaf of $T'$ (including $y$) is type-$0$.  This contradicts the maximality of $T$.
\end{CASE}

In all situations we reach a contradiction, so we conclude that $V(T)=V(G)$, as required.
\end{proof}

\bibliographystyle{plain}
\bibliography{References}

\begin{thebibliography}{1}

\bibitem{AGLY24arxiv}
J.~Ai, Z.~Gao, X.~Liu, and J.~Yue.
\newblock A short note on spanning even trees, 2024.
\newblock \href{https://arxiv.org/abs/2408.07056}{\texttt{arXiv:2408.07056}}.

\bibitem{Bae37}
F.~B\"{a}bler.
\newblock \"{U}ber die {Z}erlegung regul\"{a}rer {S}treckencomplexe ungerader {O}rdnung.
\newblock {\em Comment. Math. Helv.}, 10:275--287, 1937.

\bibitem{EHSS24arxiv}
M.~N. Ellingham, Y.~Huang, S.~Shan, and S.~\v{S}pacapan.
\newblock Spanning weakly even trees of graphs, 2024.
\newblock \href{https://arxiv.org/abs/2409.15522v1}{\texttt{arXiv:2409.15522v1}}.

\bibitem{MR1670596}
D.~Hanson, C.~O.~M. Loten, and B.~Toft.
\newblock On interval colourings of bi-regular bipartite graphs.
\newblock {\em Ars Combin.}, 50:23--32, 1998.

\bibitem{Jackson-Yoshimoto}
B.~Jackson and K.~Yoshimoto.
\newblock Spanning even trees of graphs.
\newblock {\em J. Graph Theory}, 107(1):95--106, 2024.

\bibitem{Pet1891}
J.~Petersen.
\newblock {D}ie {T}heorie der regul\"{a}ren {G}raphs.
\newblock {\em Acta Math.}, 15:193--220, 1891.

\end{thebibliography}

\end{document}